\documentclass[a4paper]{article}

\usepackage[utf8]{inputenc}
\usepackage[T1]{fontenc}

\usepackage{mathbbol}
\usepackage{amsfonts}
\usepackage{amssymb}
\DeclareSymbolFontAlphabet{\mathbb}{AMSb}
\DeclareSymbolFontAlphabet{\mathbbl}{bbold}
\usepackage{amsthm}

\usepackage{mathtools}
\usepackage{mathbbol}
\usepackage{enumitem}
\usepackage{verbatim}
\usepackage{graphicx}

\usepackage{tikz-cd}

\usepackage[colorlinks]{hyperref}    
\usepackage{cleveref}

\theoremstyle{plain}
\newtheorem{thm}{Theorem}

\newtheorem{lemma}[thm]{Lemma}
\newtheorem{cor}[thm]{Corollary}

\newtheorem{remark}[thm]{Remark}
\newtheorem*{remark*}{Remark}
\newtheorem*{falseexample*}{Naive reasoning}

\theoremstyle{definition}
\newtheorem{definition}[thm]{Definition}

\newtheorem{notation}[thm]{Notation}

\newcommand{\nat}{\mathbb{N}}
\newcommand{\integ}{\mathbb{Z}}
\newcommand{\real}{\mathbb{R}}
\newcommand{\compl}{\mathbb{C}}

\DeclareMathOperator{\CS}{CS}

\newcommand{\sphere}[1]{\mathbb{S}^{#1}}
\newcommand{\cercle}{\sphere{1}}

\newcommand{\spheretwon}{\mathbb{S}^{2n}}

\newcommand{\sumrisqvarphiidiff}{\sum_{i=1}^{n} r_i^2 d\varphi_i}
\newcommand{\sumzinormsquare}{\sum_{i=1}^{n}\vert z_i \vert^2}

\DeclareMathOperator{\Id}{Id}
\DeclareMathOperator{\image}{Im}

\newcommand{\diff}[1]{\mathrm{Diff}\left(#1\right)}
\newcommand{\diffVxi}{\diff{V,\xi}}
\newcommand{\diffV}{\diff{V}}

\newcommand{\ContStr}[1]{\mathrm{Cont}\left(#1\right)}
\newcommand{\ContStrV}{\ContStr{V}}

\newcommand{\lineW}{\overline{W}}
\newcommand{\double}{\mathcal{D}}
\newcommand{\doubleW}{\double W}
\newcommand{\doublefunW}[1]{\double_{#1}W}

\newcommand{\What}{\widehat{W}}

\newcommand{\GL}{GL}

\newcommand{\family}{\mathfrak{L}}

\newcommand{\lie}{\mathcal{L}}

\newcommand{\difff}[1]{\mathrm{FDiff}\left(#1\right)}

\newcommand{\diffVxif}{\difff{V,\xi}}

\newcommand{\xistab}{\xi_{stab}}

\DeclarePairedDelimiter\abs{\lvert}{\rvert}%
\title{Examples of contact mapping classes \\ of infinite order in all dimensions}
\author{Fabio Gironella}
\date{}

\begin{document}

\def\co{\colon\thinspace}
\def\coeq{\coloneqq\thinspace}

\maketitle

\begin{abstract}
We give examples of tight high dimensional contact manifolds admitting a contactomorphism whose powers are all smoothly isotopic but not contact-isotopic to the identity. 
This is a generalization of an observation in dimension $3$ by Gompf, also reused by Ding and Geiges.
\end{abstract}

\section{Introduction}
\label{SecIntro}

In this paper we study the topology of the space of contactomorphisms $\diffVxi$ of a given contact manifold $(V,\xi)$, comparing it with the one of the space of diffeomorphisms $\diffV$ of the underlying smooth manifold $V$. 
\\
It is known that the space $\ContStrV$ of contact structures on $V$ plays an important role in the study of the relations between $\diffVxi$ and $\diffV$. 
Indeed, if $V$ is a closed manifold, then the map $\diffV\rightarrow\ContStrV$, defined by $\phi\mapsto \phi_*\xi$, is a locally-trivial fibration with fiber $\diffVxi$; this essentially follows from (the proof of) Gray's theorem, as explained for instance in \cite{MasGir15,MasFibrNotes}. (See also \cite{GeiGon04}, in which it is proved that the map is a Serre fibration, which is enough for this discussion.)
This fibration induces a long exact sequence of homotopy groups
\begin{equation}
	\label{eq:long_exact_seq}
	\ldots \rightarrow \pi_{k+1}\left(\ContStrV\right)\rightarrow \pi_k\left(\diffVxi\right) \xrightarrow{j_*} \pi_k\left(\diffV\right)\rightarrow \pi_k\left(\ContStrV\right)\rightarrow\ldots
\end{equation}
where $j_*:\pi_k\left(\diffVxi\right)\rightarrow\pi_k\left(\diffV\right)$ is the map induced on the homotopy groups by the natural inclusion $j:\diffVxi\rightarrow\diffV$.
(Here, the homotopy groups of $\diffV$ and $\diffVxi$ are considered with base point $\Id$, whereas those of $\ContStrV$ are considered based at $\xi$; for notational ease, this natural choice of base points is suppressed from the notation.)

The higher (i.e.\ $k\geq2$) homotopy groups $\pi_k(\ContStrV,\xi)$ have not been studied much in the literature.
To our knowledge, the only works dealing with these higher homotopy groups are \cite{Bou06,CasSpa16,FerGir19},
which use, respectively, contact homology, some algebraic topology and the $h$-principle from \cite{BorEliMur15} in order to study the groups $\pi_k\left(\ContStrV,\xi\right)$ for some particular contact manifolds $(V,\xi)$ and some $k>1$.

In the rest of the literature, the focus is typically on the study of the map $j_*:\pi_0\left(\diffVxi\right)\rightarrow\pi_0\left(\diffV\right)$ induced by $j$ on the space of connected components, which is of course deeply related, via the exact sequence above, to the study of $\pi_1\left(\ContStrV,\xi\right)$.
The results available so far in this direction also consist in concrete examples of contact manifolds $(V,\xi)$ where, thanks to the specific geometry of the underlying manifold $V$, one can use techniques from both contact geometry, such as convex surface theory in the tight and overtwisted $3$-dimensional case and holomorphic curves in the tight high-dimensional case, and from algebraic topology, in the overtwisted high dimensional case, to obtain results  on the map $j_*\vert_{\pi_{0}}$ and on the fundamental group of $\ContStrV$. 
\\
For instance, one can find in the literature several examples of contact manifolds $(V,\xi)$ for which $\ker\left(j_*\vert_{\pi_{0}}\right)$ is non-trivial; 
the interested reader can consult \cite{Gom98,Gir01,GeiGon04,Bou06,DinGei10,GeiKlu14,MasGir15} and \cite{Dym01,Vog16} for, respectively the tight and overtwisted $3$-dimensional cases, and \cite{Bou06,MasNie16,LanZap15} and \cite{My17a} for, respectively, the tight and overtwisted higher-dimensional cases.
Notice that the examples in \cite{Bou06,MasNie16,LanZap15} are tight according to the definition of overtwistedness in higher dimensions in \cite{BorEliMur15}, which generalizes the $3$-dimensional one given in \cite{Eli89}.

This paper focuses more precisely on the problem of the existence of infinite cyclic subgroups in $\ker(j_*\vert_{\pi_{0}})$.
To our knowledge, the only known example of such phenomenon is given in \cite{Gom98,DinGei10}. 
More precisely, in \cite[Remark at page 642]{Gom98}, using rotation numbers, Gompf argues that $\sphere{2}\times\cercle$, equipped with its unique (up to isotopy) tight contact structure $\xi_{std}$, has a contact mapping class of infinite order.
Then, starting from Gompf's remark, Ding and Geiges prove in \cite{DinGei10} that $\ker(j_*\vert_{\pi_{0}})$ and $\pi_1(\ContStr{\sphere{2}\times\cercle},\xi_{std})$ are actually both isomorphic to $\integ$.

Our aim is to give explicit examples of high-dimensional tight manifolds that admit an element of infinite order in $\ker(j_*\vert_{\pi_{0}})$.
This is achieved by first exhibiting elements of infinite order in $\pi_0(\diffVxi)$ for $V$ given by the product of the double $\doubleW$ of a stabilized Weinstein domain $W$ and the circle $\cercle$, equipped with a natural fillable contact structure $\xi$ on it.
For some particular choices of $W$, these infinite-order elements of $\pi_0(\diffVxi)$ can moreover be shown to be in $\ker(j_*\vert_{\pi_{0}})$.
\\
More precisely, we start by analyzing the following general situation.
Consider a Weinstein manifold $(F^{2n-2},\omega_F,Z_F,\psi_F)$; namely, $F$ has only positive boundary component $\partial F = \partial_+ F$, $\omega_F$ is a symplectic form on $F$, $\psi_F\colon F \rightarrow \real$ is an exhausting (i.e.\ proper and bounded from below) Morse function, such that $\psi_F$ is constant on $\partial_+ F$ and $Z_F$ is a complete Liouville vector field for $\omega_F$, which is also gradient-like for $\psi_F$. 
Consider then the stabilization $(F\times\compl,\omega_F\oplus\omega_0,Z_F+Z_0,\psi_F+\abs{.}^2_\compl)$, where $\omega_0=rdr\wedge d\varphi$ and $Z_0 = \frac{1}{2}r\partial_r$, using coordinates $z=re^{i\varphi}\in\compl$. 
Suppose that $c> \min \psi_F$ is a regular value of $\psi\coeq\psi_F+\abs{.}^2_\compl$ and let $W$ be the compact domain $\psi^{-1}((-\infty,c])$.
We also assume that there is an almost complex structure $J_F$ on $F$, which is tamed by $\omega_F$ and such that $(TF,J_F)$ is trivial as complex bundle over $F$.
\\
Consider now the Weinstein manifold $(F\times\compl\times\real\times\cercle,\omega',Z',\psi')$, where, using coordinates $(s,\theta)\in\real\times\cercle$, $\omega' = \omega_F + \omega_0 + ds\wedge d\theta$, $Z' = Z_F + Z_0 + s\partial_s$ and $\psi'(p,z,s,\theta) = \psi(p,z) + s^2$.
The preimage $(\psi')^{-1}(c)$, which is diffeomorphic to the product of the double $\doubleW\coeq W\cup_{\partial W} \lineW$ of $W$ and $\cercle$, is naturally equipped with the contact structure $\xi=\ker\alpha$, where $\alpha = (\iota_{Z'}\omega')\vert_{\doubleW\times\cercle}$.
Moreover, the diffeomorphism of $F\times\compl\times\real\times\cercle$ given by $(q,z,s,\theta)\mapsto (q,e^{i\theta}z,s,\theta)$ restricts to a well defined diffeomorphism $\Psi$ of $\doubleW\times\cercle$.
\\
We will then prove the following result:
\begin{thm}
	\label{ThmMain}
	In the setting described above (recall, in particular, that $(TF,J_F)$ is assumed to be trivial), the diffeomorphism $\Psi$ of $\doubleW\times\cercle$ is smoothly isotopic to a contactomorphism $\Psi_c$ of $(\doubleW\times\cercle,\xi)$ such that, for each integer $k\neq0$, its $k$-th iterate is not contact-isotopic to the identity.
\end{thm}

A direct application of \Cref{ThmMain} with $F=\compl^{n-1}$, $\omega_F= \sum_{i=1}^{n-1}r_i dr_i\wedge d\varphi_i$, $Z_F =\frac{1}{2} \sum_{i=1}^{n-1} r_i\partial_{r_{i}}$, $\psi_F(z_1,\ldots,z_{n-1})=r_1^2+\ldots + r_{n-1}^2$ and $c=1$, where we use polar coordinates $(z_1=r_1e^{i\varphi_{1}},\ldots,z_{n-1}=r_{n-1}e^{i\varphi_{n-1}})$ on $F=\compl^{n-1}$, gives the following generalization of the observation in \cite{Gom98} to higher dimensions:
\begin{cor}
	\label{CorMain}
	Let $(z_1,\ldots,z_n,s,\theta)$ be coordinates on $\compl^n\times\real\times\cercle=\real^{2n+1}\times\cercle$ and $\xi$ be the tight contact structure on $V\coeq\spheretwon\times\cercle$ defined by the restriction of $\lambda = sd\theta + \frac{1}{2}\sumrisqvarphiidiff $ on $\real^{2n+1}\times\cercle$ to $\spheretwon\times\cercle=\{s^2 + \sumzinormsquare=1\}$.
	Consider now the diffeomorphism $\Psi$ of $\spheretwon\times\cercle$ given by the restriction of 
	\begin{align*}
	\real^{2n+1}\times\cercle & \longrightarrow\real^{2n+1}\times\cercle \\
	(z_1,\ldots,z_n,s,\theta)& \mapsto (z_1,\ldots,z_{n-1},e^{i\theta}z_n,s,\theta)
	\end{align*}
	\\
	Then, $\Psi$ is smoothly isotopic to a contactomorphism $\Psi_c$ of $(V,\xi)$ such that $[\Psi_c^2]$ generates an infinite cyclic subgroup of $\ker\left(\pi_0\diff{V,\xi}\rightarrow\pi_0\diff{V}\right)$.
\end{cor}
Notice that each even power of $\Psi_c$ in \Cref{CorMain} is indeed smoothly isotopic to the identity: 
because the fundamental group of $SO(m)$ is isomorphic to $\integ_2$ for all $m\geq3$, 
there is, for all $k\in\nat$, a smooth isotopy of $\spheretwon\times\cercle$, (globally) preserving each submanifold $\spheretwon\times\{pt\}$, from $\Psi^{2k}$ to the identity; in particular, $\Psi_c^{2k}$ is also smoothly isotopic to the identity.

Analogously to \Cref{CorMain}, \Cref{ThmMain} can be applied to the case of $F=T^*\mathbb{T}^n$, $\omega_F=\sum_{i=1}^n dp_i\wedge dq_i$, 
$Z_F=\sum_{i=1}^n p_i\partial_{p_{i}}$, $c=1$ and $\psi_F(q_i,p_i)=\sum_{i=1}^n p_i^2$ (perturbed to a Morse function with a perturbation supported on a neighborhood of $\psi_F^{-1}(0)$).
This gives, for each $n\geq1$, another explicit example of tight $(V^{2n+1},\xi)$ such that $\ker\left(\pi_0\diff{V,\xi}\rightarrow\pi_0\diff{V}\right)$ has an infinite cyclic subgroup.
\\
Indeed, each even power of the diffeomorphism $\Psi$ is smoothly isotopic to the identity. This follows from the facts that $T^*\mathbb{T}^n \simeq \mathbb{T}^n\times\real^n$, that $\doubleW\times\cercle\simeq\mathbb{T}^n\times \sphere{n+2}\times\cercle$ and that, for each $\theta\in\cercle$, $\Psi\co\mathbb{T}^n\times \sphere{n+2}\times\cercle\rightarrow\mathbb{T}^n\times \sphere{n+2}\times\cercle$ acts trivially on the first and thirds factors and as a rotation of angle $\theta$ around a given axis on each $\{pt\}\times\sphere{n+2}\times\{\theta\}$; because $\pi_1(SO(m))\simeq\integ_2$ for each $m\geq3$, we can then conclude, as done in the case of \Cref{CorMain}, that $\Psi_c^{2k}$ is smoothly isotopic to the identity, for each $k\neq1$.

\begin{remark}
	\label{rmk:formal}
	To be precise, the actual content of \Cref{ThmMain}, and hence of \Cref{CorMain}, is that $\Psi_c$ is not homotopic to $\Id$ among \emph{formal contactomorphisms} (cf.\ \Cref{rmk:formal_bis}). 
	In analogy with the case of formal isocontact embeddings in \cite[Section 12.3]{EliMisBook}, by \emph{formal contactomorphism} we mean here a pair $(f,F_t)$ where $f\in\diffV$ and $(F_t)_{t\in[0,1]}$ is a path of fiber-wise injective morphisms $TV\to TV$, each lifting $f$, which starts at $F_0=d f$ and ends at a certain $F_1$ preserving both $\xi$ and the conformal symplectic structure $\CS_\xi$, i.e.\ $F_1(\xi)=\xi$ and $(F_1)_*\CS_{\xi} = \CS_{\xi}$. 
	We denote the space of such pairs by $\diffVxif$.
	(Notice that there is an alternative, weaker but equally natural, notion of formal contactomorphism in the literature, which gives a long exact sequence analogous to the one in \Cref{eq:long_exact_seq} but with the space of almost contact structures; see \cite[Appendix A]{CKS18}.)
	In particular, we point out that the question of finding (explicit) examples of contactomorphisms which are of infinite order in $\pi_0(\diffVxi)$ but are formally trivial, i.e.\ trivial in $\pi_0(\diffVxif)$, is still open.
\end{remark}

\section*{Acknowledgements}

This work is part of my PhD thesis \cite{MyPhDThesis}, written at Centre de mathématiques Laurent Schwartz of École polytechnique (Palaiseau, France).
A revised version has been carried out while at Alfred Renyi Institute in Budapest, supported by the grant NKFIH KKP 126683. I am currently supported by the ERC Grant ``Transholomorphic''.

I would like to thank my PhD advisor Patrick Massot, for many useful remarks and comments which greatly improved this manuscript, as well as Eduardo Fernández, for useful discussions on the spaces of formal contactomorphisms and formal Legendrian embeddings.
I am also extremely grateful to Sylvain Courte for useful discussions on deformations of Weinstein structures to Stein structures, which were relevant for the previous version of this manuscript, as well as to Emmanuel Giroux, for letting me include, also in the previous version of this text, some extracts of his private notes \cite{GirouxPersNotes}, concerning deformation of Weinstein structures to almost Stein structures.
Lastly, I wish to thank the anonymous referee for very useful comments and suggestion, that led in particular to many simplifications in the structure of the proof, and thus greatly improved the presentation and the readability of the whole paper.

\section{Preliminaries}

In \Cref{SubSecProdDoubleLiouvDomWithCircle} we describe how, given a Liouville domain $W$, one can naturally construct an explicit Liouville manifold having $\doubleW\times\cercle$ as convex boundary, as well as contactomorphisms of the latter; this will then be used in the case of Weinstein domains in the proof of \Cref{ThmMain}.
\\
\Cref{SubSecFamLegBases} introduces a simple invariant of homotopical nature for contact isotopy (actually, formal contact isotopy) classes of contactomorphisms, namely \emph{families of Lagrangian frames}.


\subsection{Product of doubled Liouville domains and $\cercle$}
\label{SubSecProdDoubleLiouvDomWithCircle}

Let $\What$ be a smooth manifold of dimension $m$ and $f\co \What\rightarrow \real$ be a proper and bounded from below function which is also a \emph{regular equation} of a (cooriented) hypersurface $M\subset \What$, i.e. $f$ is transverse to $0$ and satisfies $M=f^{-1}(0)$ (with coorientation). 
We denote by $W$ the compact submanifold $f^{-1}((-\infty,0])$ of $\What$.
\begin{definition}
	\label{DefDoubleFun}
	We denote by $\doublefunW{f}$ the smooth manifold given by $\{(p,s) \in \What\times\real \,\vert\, s^2+f(p)=0\}$, and we call it \emph{$f$-double} of $W$.
\end{definition}
Notice that $\doublefunW{f}$ is just obtained by gluing two copies of $W$, embedded as graphs of $W$ inside $W\times\real$, one in the region $\{s\geq0\}$ and one in the region $\{s\leq0\}$, along their common boundary, which lies in $\{s=0\}=W\times\{0\}$.
This justifies the nomenclature ``$f$-double''.

We also point out that $\doublefunW{f}$ is indeed a smooth submanifold of $\What\times\real$, because the function $\What\times\real\rightarrow\real$ given by $(p,s)\mapsto s^2 + f(p)$ is transverse to $0$.
\\
Indeed, one can always find a vector field $Z$ on $\What$ which is \emph{boundary-gradient-like} for $f$, i.e which satisfies $df(Z)\geq 0$ everywhere on $\What$ and $df(Z)>0$ along $M=f^{-1}(0)$. 
More precisely, there is a vector field $Z'$ on a neighborhood $U$ of the (cooriented) hypersurface $M$ such that $df(Z')>0$ on $U$, and we can choose $Z$ to be $Z'$ multiplied by a non-negative cutoff function $\chi$ supported in $U$.
Then, $d(s^2+f)(s\partial_s + Z) = 2s^2 + df(Z)>0$ along $\doublefunW{f}\subset \What\times \real$, which shows that $\doublefunW{f}$ is a regular hypersurface.

\begin{notation}
	\label{NotDouble}
If $f\co\What\rightarrow\real$, we denote by $f^\double\co \What\times\real\rightarrow\real$ the function $f^\double(p,s)=s^2+f(p)$.
In particular, if $f$ is an equation of the hypersurface $M\subset \What$, then $f^\double$ defines the hypersurface $\doubleW\subset\What\times\real$, as shown above.
\\
In a similar way, if $Z$ is a vector field on $\What$, we denote by $Z^\double$ the vector field $Z+s\partial_s$ on $\What\times\real_s$;  if $Z$ is boundary-gradient-like for $f$, then so is $Z^\double$ for $f^\double$.
\end{notation}
 
We now describe how $\doublefunW{f}$ actually depends on the choice of $f$. 
For this, we first state a general uniqueness lemma. 
\begin{lemma}
	\label{LemmaGenUniqSmoothDouble}
	Let $Y$ be a smooth manifold and $(g_t)_{t\in[0,1]}$ be a smooth family of functions $g_t\co Y\to\real$ such that $0$ is a regular value for each $g_t$. 
	Additionally, assume that each $N_t\coeq g_t^{-1}(0)$ is compact and that there is a vector field $V$ on $Y$ which is boundary-gradient-like for every $g_t$ with respect to $N_t$ (i.e.\ $dg_t(Z)>0$ along $N_t$).
	Then the flow $\psi_{X_t}^t$ of
	\begin{equation*}
	X_t\coeq\frac{\dot{g_t}}{d g_t(V)}\,  V
	\end{equation*}
	satisfies $\psi_{X_{t}}^t(N_t)=N_0$.
	In particular, all the $N_t$'s are diffeomorphic.
\end{lemma}
\begin{proof}
	The smooth function $G\co Y\times [0,1]\rightarrow \real$, given by $G(y,t)= g_t(y)$, is transverse to $0$: indeed, $dG(V)>0$ along $G^{-1}(0)=\bigcup_t N_t\times\{t\}$.
	Then, $G^{-1}(0)$ is a smooth submanifold of $Y\times [0,1]$, which is moreover contained in $Y\times[0,1]\setminus\{(y,t)\vert dg_t(V)=0\}$.
	Moreover, the (well defined on $\image G$) vector field $-\partial_t + X_t$
	is tangent to $G^{-1}(0)$ and its flow at time $1$ restricts to a diffeomorphism from $G^{-1}(0)\cap \left(Y\times\{1\}\right)=N_1$ to $G^{-1}(0)\cap \left(Y\times\{0\}\right)=N_0$, as desired.
\end{proof}

\noindent
In our setting of the $f$-double, we get the following:
\begin{cor}
	\label{LemmaUniqSmoothDouble}
	Let $f_0,f_1\co \What\rightarrow\real$ be proper, bounded from below and both regular equations for $M$, and $Z$ be a vector field on $\What$ which is boundary-gradient-like for both $f_0$ and $f_1$.
	Then, $\doublefunW{f_0}$ and $\doublefunW{f_1}$ are diffeomorphic.
\end{cor}
\begin{proof}
	Apply \Cref{LemmaGenUniqSmoothDouble} with $Y=\What\times\real_s$, $g_t=tf_1^\double +(1-t)f_0^\double$ and $V=Z$.
\end{proof}

\noindent
Notice that if $f_0,f_1\co \What\rightarrow\real$ are two regular equations for $M$, then there always is a vector field $Z$ on $\What$ which is boundary-gradient-like for both $f_0$ and $f_1$; this can be proven as done above in the case of a single regular equation.
\Cref{LemmaUniqSmoothDouble} then tells that $\doublefunW{f}$ does not depend on $f$, up to diffeomorphism.
By a slight abuse of notation, we may hence write $\doubleW$ and simply talk about the \emph{double} of $W$.
\\

Let now $(\What^{2n},\lambda)$ be a Liouville manifold and denote by $Z$ its Liouville vector field. Consider also a smooth proper function $f\co \What \rightarrow \real$, bounded from below and such that $Z$ is boundary-gradient-like for $f$.
Denote by $W$ the (compact) submanifold $f^{-1}((-\infty,0])$ of $\What$. 
Notice that $(M,\eta=\ker(\lambda\vert_M))$ is a contact manifold and that $(W,\lambda)$ is a Liouville filling of it. 
\\
Consider now the Liouville manifold $(\What\times\real_s\times\cercle_\theta,\lambda+sd\theta)$, where $\real_s$ and $\cercle_\theta$ denote the manifolds $\real$ and $\cercle$ with coordinates $s$ and $\theta$ respectively. 
Notice that the vector field $Z^\double=Z+s\partial_s$ and the function $f^\double$ can naturally be seen on $\What\times\real\times\cercle$. 
Moreover, $Z^\double$ is Liouville for $\lambda + sd\theta$ and transverse to $\doublefunW{f}\times\cercle=\{f^\double=0\}\subset \What\times\real\times\cercle$; in particular, $\alpha_f\coeq (\lambda + sd\theta)\vert_{\doublefunW{f}\times\cercle}$ is a contact form on $\doublefunW{f}\times\cercle$.
In analogy with \Cref{NotDouble}, we will also denote the Liouville form $\lambda+sd\theta$ on $\What\times\real\times\cercle$ by $\lambda^\double$ in the following.

As in the case of the smooth double, we now describe how $(\doublefunW{f}\times\cercle,\ker\alpha_f)$ depends on the specific choice of $f$.
For this, we first state a more general uniqueness property: 
\begin{lemma}
	\label{LemmaGenUniqContStrDouble}
	Let $(Y,\lambda_Y)$ be a Liouville manifold and denote by $V$ the Liouville vector field associated to $\lambda_Y$.
	Consider also a smooth family $(g_t)_{t\in[0,1]}$ of functions $g_t\co Y\to\real$ such that $0$ is a regular value for each $g_t$, $N_t\coeq g_t^{-1}(0)$ is compact, and $V$ is boundary-gradient-like for every $g_t$ with respect to $N_t$;
	in particular, $\xi_t\coeq\ker(\lambda_Y\vert_{N_t})$ is a contact form on $N_t$.
	Then the flow $\psi_{X_t}^t$ of
	\begin{equation*}
	X_t\coeq\frac{\dot{g_t}}{d g_t(V)}\,  V
	\end{equation*}
	satisfies $\psi_{X_{t}}^t(N_t)=N_0$ and $(\psi_{X_{t}}^t\vert_{N_t})_*\xi_t=\xi_0$.
	In particular, all the $(N_t,\xi_t)$'s are contactomorphic.
\end{lemma}
\begin{proof}
	According to \Cref{LemmaGenUniqSmoothDouble}, it's enough to prove that the flow $\psi_{X_t}^t$ of $X_t$ preserves $\ker(\lambda)$ (on its domain of definition).
	An explicit computation shows that $\mathcal{L}_{X_{t}}(\lambda) = \frac{\dot{g_t}}{d g_t(V)} \lambda$. 
	In particular, one can check that $(\psi_{X_t}^t)^{*}\lambda = h_t \lambda$ with 
	\begin{equation*}
	h_t= \exp\left(\int_{0}^{t}\frac{\dot{g_s}}{d g_s(V)}\circ \psi_{X_{s}}^s ds\right) \text{ ,}
	\end{equation*}
	well defined on a neighborhood of $N_0$ and with values in $\real_{>0}$, as desired.
\end{proof}

\noindent
Going back to our particular setting, we get:
\begin{cor}
	\label{LemmaUniqContStrDouble}
	Let $f_0,f_1\co \What\rightarrow\real$ be proper, bounded from below and both regular equations for $M$, and $Z$ be a Liouville vector field on $\What\times\real\times\cercle$ which is boundary-gradient-like for both $f_0$ and $f_1$.
	Then, $(\doublefunW{f_0}\times\cercle,\ker(\alpha_{f_{0}}))$ and $(\doublefunW{f_1}\times\cercle,\ker(\alpha_{f_{1}}))$ are contactomorphic.
\end{cor}
\begin{proof}
	Simply apply \Cref{LemmaGenUniqContStrDouble} with $(Y,\lambda_Y)=(\What\times\real_s\times\cercle_\theta, \lambda + sd\theta)$, $g_t=tf_1^\double +(1-t)f_0^\double$ and $V=Z$.
\end{proof}

\noindent
We will hence drop the $f$ from the notation and just denote $(\doubleW\times\cercle,\ker\alpha)$.

\begin{remark*}
In \cite{GeiSti10}, Geiges and Stipsicz construct, more generally, contact forms on $(W_1\cup_M W_2)\times \cercle$, where $(W_1,\lambda_1)$ and $(W_2,\lambda_2)$ are Liouville domains with the same (strict) contact boundary $(M,\alpha)$.
The contact structure they obtain in the particular case where $W_1=W_2$ and $\lambda_1=\lambda_2$ (and $\partial W_1$ identified with $\partial W_2$ via the identity) is the same, up to isotopy, as the contact structure on $\doubleW\times\cercle$ that we described above. 
\\
Even though the construction described in this paper is less general (as it only covers the case $W_1=W_2$), it has the advantage of involving a natural Liouville filling of the strict contact manifold $(\doubleW\times\cercle,\alpha)$, which will be useful in \Cref{SecProofs}.
We also point out that one cannot always expect a presentation involving a symplectic filling for the construction in \cite{GeiSti10}.
For instance, in the case $W_1=D^2$ and $W_2=\Sigma_g\setminus D^2$, where $\Sigma_g$ is a closed surface with genus $g\neq0$, the theory of convex surfaces by Giroux tells that the contact structure on $(W_1\cup_{\cercle}W_2)\times\cercle$ obtained as in \cite{GeiSti10} is overtwisted: indeed, it is the unique $\cercle$-invariant contact structure on $\Sigma_g\times\cercle$ such that each $\Sigma_g\times\{pt\}$ is a convex surface with dividing set consisting of a homotopically trivial circle.
\end{remark*}

We now describe an explicit natural way to construct (strict) contactomorphisms of $(\doubleW\times\cercle,\xi\coeq\ker\alpha)$.
\\
Consider a smooth action $\rho\co \cercle_\theta\times\What\rightarrow\What$ by diffeomorphisms $\varphi_\theta\coeq \rho(\theta,.)$ of $\What$, each of which preserves both $\lambda$ and $f\co\What\rightarrow\real$.
(Notice that we do not assume that $\rho$ restricts to the trivial action by the identity on $M=\partial W$.) 
Let also $X$ be the infinitesimal generator of $\rho$, i.e.\ the vector field on $\What$ given by 
\begin{equation*}
	X(p)=\left.\frac{d}{d\theta}\right\vert_{\theta=0}\varphi_{\theta}(p) \text{ ,}
\end{equation*}
where we interpret $\cercle=\real/\integ$ and $\varphi_.(p)\co(-\epsilon,\epsilon)\to\What$ as a curve on $\What$ passing through $p$ at time $\theta=0$.
Lastly, consider the diffeomorphism $\Psi\co  \doubleW  \times\cercle \rightarrow\doubleW\times\cercle$ given by the restriction of 
\begin{equation*}
	\widehat{\Psi}\co\What\times\real\times\cercle\rightarrow\What\times\real\times\cercle \text{ ,} \quad
	(p,s,\theta)\mapsto(\varphi_\theta(p),s,\theta)
\end{equation*}
to $\doubleW  \times\cercle$; notice that this restriction is well defined because $\varphi_\theta$ preserves $f$.

\begin{lemma}
	\label{LemmaContactomorphismProdWithCircle}
	The flow $\psi_Y^t$ of the vector field
	\begin{equation*}
	Y = \frac{\lambda(X)}{ d f^\double(Z^\double)}\, \left( 2s \, Z - df(Z)\, \partial_s\right) 
	\end{equation*}
	gives a smooth isotopy $\Psi\circ\psi_Y^t$ from $\Psi = \Psi\circ\psi_Y^0$ to a contactomorphism $\Psi_c\coeq \Psi\circ\psi_Y^1$ of $(\doubleW\times\cercle,\xi=\ker\alpha)$.
\end{lemma}

Notice that $Y$, well defined as a vector field on $\What\times\real\times\cercle\setminus \{s=0,df(Z)=0\}$, is indeed tangent to (the $\doubleW$ factor of) $\doubleW\times\cercle$.

\begin{proof}[Proof (\Cref{LemmaContactomorphismProdWithCircle})]	
	For notational ease, 	let $h\coeq\lambda(X)$, defined on $\What$.
	Notice that $\lie_{X}\lambda = 0$, as $\varphi_\theta^*\lambda=\lambda$ for each $\theta\in\cercle$. 
	In particular, $dh = - \iota_{X}d\lambda$ and,
	evaluating on the Liouville vector field $Z$, one gets $dh(Z)=h$.
	
	We now want to compute the pullback $\Psi^*\left[\lambda^\double\vert_{\doubleW\times\cercle}\right]$.
	\\
	First, notice that $(\widehat{\Psi}^*\lambda^\double)(\partial_s)=\lambda^\double_{\widehat{\Psi}(.)}(d\widehat{\Psi}(\partial_s))=\lambda^\double_{\widehat{\Psi}(.)}(\partial_s)=0$.
	Moreover, for any vector field $V$ on $\What\times\real_s\times\cercle_\theta$ tangent to the $\What$ factor, one can compute
	\begin{align*}
		(\widehat{\Psi}^*\lambda^\double)_{(p,s,\theta)}(V(p,s,\theta)) & = \lambda^\double_{\widehat{\Psi}(p,s,\theta)}(d_{(p,s,\theta)}\widehat{\Psi}(V(p,s,\theta))) \\
		& = \lambda_{\varphi_\theta(p)}(d_{p}\varphi_\theta(V(p,s,\theta))) \\
		& = (\varphi_\theta^*\lambda)_p(V(p,s,\theta)) \\
		& = \lambda_p(V(p,s,\theta))\text{ .}
	\end{align*}
	Lastly, one can similarly compute:
	\begin{align*}
	(\widehat{\Psi}^*\lambda^\double)_{(p,s,\theta)}(\partial_\theta) & = \lambda^\double_{\widehat{\Psi}(p,s,\theta)}(d_{(p,s,\theta)}\widehat{\Psi}(\partial_\theta)) \\
	& = (\lambda +  s d\theta)_{(\varphi_\theta(p),s,\theta)}(X(\varphi_\theta(p)) + \partial_\theta) \\
	& = \lambda_{\varphi_\theta(p)}(X(\varphi_\theta)) + s \\
	& \overset{(*)}{=} \lambda_p(X(p)) + s\text{ ,}
	\end{align*}
	where $(*)$ used the fact that $\varphi_\theta^*(\iota_X\lambda)= \iota_{(\varphi_\theta^*X)}(\varphi_\theta^*\lambda) = \iota_X \lambda$, as $\varphi_\theta$ preserves both $X$ and $\lambda$ for each $\theta\in\cercle$.
	In conclusion, $\widehat{\Psi}^*\lambda^\double=\lambda +\left(s+h\right)d\theta$, hence
	\begin{equation*}
		\Psi^*\left[\lambda^\double\vert_{T(\doubleW\times\cercle)}\right]=\left[\lambda +\left(s+h\right)d\theta\right]\vert_{T(\doubleW\times\cercle)} \text{ .}
	\end{equation*}
	
	We now describe a simple homotopy of contact forms between $\alpha$ and $\Psi^*\alpha$.
	\\
	For all $t\in[0,1]$, denote the $1$-form $\lambda + (s+th)d\theta$ on $\What\times\real\times\cercle$ by $\lambda^\double_t$, and its restriction to $\doubleW\times\cercle$ by $\alpha_t$. 
	An explicit computation shows that, for each $t\in[0,1]$,  $\lambda^\double_t$ is a Liouville form with corresponding Liouville vector field $Z^\double$ (independent of $t$), which is hence transverse to $\doubleW\times\cercle$.
	In particular, for each $t\in[0,1]$, $\alpha_t$ is a contact form on $\doubleW\times\cercle$.

	Lastly, we need to prove that this homotopy of contact forms $\alpha_t$ is realized by the vector field $Y$ as in the statement.
	According to (the proof of) Gray's theorem, the flow of the (a priori time-dependent) vector field $X_t$ such that $\alpha_t(X_t)=0$ and $\iota_{X_{t}}d\alpha_t\vert_{\ker\alpha_t}=-\dot{\alpha}_t\vert_{\ker\alpha_t}$ gives an isotopy that pulls back $\ker\alpha_t$ to $\ker\alpha_0$ (see for instance the discussion after \cite[Theorem 2.2.2]{Gei08}). 
	It's hence enough to show that the vector field $Y$ in the statement verifies these two conditions.
	\\
	An explicit computation gives that $d f^\double (Y) = 0$ and $\lambda^\double_t(Y)=0$, i.e. that $Y\in \ker\alpha_t=\ker\lambda^\double_t\cap T(\doubleW\times\cercle)$.
	Moreover, we can compute
	\begin{align*}
		\iota_Y d\lambda^\double_t  & = \iota_Y(d\lambda + ds\wedge d\theta + tdh \wedge d\theta) \\
		& = \frac{h}{d f^\double(Z^\double)}\left[2s\lambda + 2ts \,d h(Z)\, d\theta - df(Z)d\theta\right] \\
		& \overset{(i)}{=} \frac{h}{d f^\double(Z^\double)}\left[2s \lambda^\double_t - 2s^2 d\theta - df(Z)d\theta \right]\\
		& \overset{(ii)}{=}  \frac{2s h}{d f^\double(Z^\double)} \lambda^\double_t -\frac{d}{dt}\lambda^\double_t \text{ ,}
	\end{align*}
	where for $(i)$ we used that $dh(Z)=h$ and for $(ii)$ we used that $df^\double (Z^\double)= 2s^2 + df(Z)$ and $\frac{d}{dt}\lambda^\double_t=h d\theta$. 
	In particular $\iota_{Y}d\alpha_t\vert_{\ker\alpha_t}=-\dot{\alpha}_t\vert_{\ker\alpha_t}$, as desired.
\end{proof}


\subsection{Families of Lagrangian frames}
\label{SubSecFamLegBases}  

Let $V$ be a smooth $(2n+1)$-manifold and $\xi$ a contact structure on $V$.
Given a compact manifold $Y^m$, we call \emph{family of Lagrangian frames} of $\xi$ indexed by $Y$ the data of a smooth map $\gamma\co Y\rightarrow V$ and, for $j=1,\ldots,n$, of smooth maps $X_j\co Y \rightarrow \xi$ such that the following diagram commutes
\begin{equation*}
\begin{tikzcd}
& \xi \ar[d] \\
Y \ar["\gamma", r] \ar["X_j",ru] & V 
\end{tikzcd}
\end{equation*}
and such that, for each $q\in Y$, the $X_1(q),\ldots, X_n(q) $ are $\real$-linearly independent and generate a Lagrangian subspace of $(\xi_p, (\CS_\xi)_p)$.
Here, $\CS_\xi$ is the natural conformal symplectic structure on $\xi$; 
in particular, $(\CS_\xi)_p$ is a conformal class of symplectic alternating forms on $\xi_p$ and, hence, has a well defined class of (isotropic and) Lagrangian subspaces.
In the following, we denote families of Lagrangian frames $(\gamma,X_1,\ldots,X_n)$ simply by $\family$.

We point out that if $f\co (V,\xi) \rightarrow (V,\xi)$ is a contactomorphism, then $f_*\family\coeq (f\circ\gamma, df (X_1),\ldots, df (X_n))$ is also a $Y$-family of Lagrangian frames of $\xi$: indeed, $f$ preserves the conformal symplectic structure $\CS_\xi$ on $\xi$. 
\\
Moreover, if $f_t\co (V,\xi) \rightarrow (V,\xi)$ is a contact-isotopy from $f_0=\Id$ to $f_1=f$, then $(f_t)_*\family$ is a path of $Y$-families of Lagrangian frames of $\xi$ from $\family$ to $f_*\family$.
In other words, we have the following obstruction to contact-isotopies:
\begin{lemma}
	\label{LemmaObstr}
	Let $f\co (V,\xi) \rightarrow (V,\xi)$ be a contactomorphism, and $\family$ a $Y$-family of Lagrangian frames  for $\xi$.
	If $f_*\family$ is not homotopic (among families of Lagrangian frames) to $\family$, then $f$ is not contact-isotopic to the identity.
\end{lemma}

What we will actually use in the proof of \Cref{ThmMain} is a stabilized version of \Cref{LemmaObstr}; here are the details.

Let $\varepsilon_\real$ be the trivial real line bundle $V\times \real \rightarrow V$, and suppose that there exists a conformal symplectic structure $\CS_{\xistab}$ on the real vector bundle $\xistab\coeq\xi\oplus\varepsilon_\real^{2k}$ extending $\CS_{\xi}$ on $\xi$ and such that $\varepsilon_\real^{2k}$ coincides with the $\CS_{\xistab}$-orthogonal complement of $\xi$ in $\xistab$;
here, $\varepsilon_\real^{2k}$ denotes the direct sum of $\varepsilon_\real$ with itself $2k$ times.
\\
Assume also that there is a complex structure $J$ on $\xistab$, which is tamed by $\CS_{\xistab}$ and such that there is an isomorphism of complex vector bundles $\Phi\co (\xistab,J) \xrightarrow{\sim} \varepsilon_\compl^{n+k}$, 
where $\varepsilon_\compl$ is the trivial complex line bundle $V\times \compl \rightarrow V$. 
\\
Notice that the property that $(\xistab,J)$ is trivial is independent of the specific choice of $J$ tamed by $\CS_{\xistab}$. 
Indeed, the space of complex structures on $\xistab$ which are tamed by $\CS_{\xistab}$ is contractible, hence $(\xi,J)$ and $(\xi,J')$ are isomorphic as complex vector bundles if $J,J'$ are both tamed by it.

Let now $\family = (\gamma,X_1,\ldots,X_n)$ be a $Y$-family of Lagrangian frames for $\xi$.
Fix also a \emph{global} Lagrangian frame $(e_1,\ldots,e_k)$ for $(\varepsilon_\real^{2k},\CS_{\xistab}\vert_{\varepsilon_\real^{2k}})$ (i.e. a $V$-family of Lagrangian frames lifting the identity map $V\to V$).
Then, the stabilization
\begin{equation*}
\family_{stab}\coeq(\gamma,X_1,\ldots,X_n,e_1\circ\gamma,\ldots,e_k\circ\gamma)
\end{equation*} 
is a $Y$-family of Lagrangian frames for $(\xistab,\CS_{\xistab})$.
In the following, we say that the family $\family_{stab}$ is the \emph{stabilization} of $\family$ \emph{via $(e_1,\ldots,e_k)$} (sometimes omitting the sections $(e_1,\ldots,e_k)$ of $\varepsilon_\real^{2k}$ if there is no ambiguity), and denote it more concisely by $\family\oplus (e_1,\ldots,e_k)$.

As $J$ is tamed by $\CS_{\xistab}$, for each $y\in Y$ the real subspace spanned by 
\[\left(X_1(y),\ldots, X_n(y), e_1\circ\gamma(y),\ldots,e_k\circ\gamma(y)\right)\]
is totally real in $((\xistab)_{\gamma(y)},J_{\gamma(y)})$.
Hence, the family of Lagrangian frames $\family_{stab}$ gives, via $\Phi$, a \emph{$Y$-family of complex frames} $\Phi_*\family_{stab}$ for $\varepsilon_\compl^{n+k}$, i.e.\ for all $y\in Y$ one has 
\begin{equation*}
	\left<\,\Phi(X_1(y)),\ldots, \Phi (X_n(y)), \Phi(e_1\circ\gamma(y)),\ldots,\Phi(e_k\circ\gamma(y))\,\right>_{\compl}=(\varepsilon_\compl^{n+k})_{\gamma(y)} = \compl^{n+k} \text{ .}
\end{equation*}
\noindent
In particular, considering, for each $y\in Y$, the linear endomorphism of $(\varepsilon_\compl^{n+k})_{\gamma(y)}=\compl^{n+k}$ obtained by sending the canonical complex framing of $\compl^{n+k}$ over $\gamma(y)$ to the one given by $\Phi_*\family_{stab}$, we obtain a smooth map $M\co Y \rightarrow\GL_{n+k}(\compl)$. 
In the following, we will say that the map $M$ is the \emph{$Y$-family of (invertible) matrices} associated (via $\Phi$) to $\family_{stab}$. 

We also point out that, given a contactomorphism $f$ of $(V,\xi)$, the family
\begin{equation*}
	(f_*\family)_{stab}=\left(\,f\circ\gamma,\, df \left(X_1\right),\ldots,\,df \left(X_n\right),\,e_1\circ f \circ \gamma,\,\ldots,\,e_k\circ f \,\circ \gamma\right)
\end{equation*}
gives another $Y$-family of invertible matrices, denoted $f_* M\co Y \rightarrow\GL_{n+k}(\compl)$.
As this can be done parametrically, we get this stabilized version of \Cref{LemmaObstr}: 
\begin{lemma}
	\label{LemmaObstrStabTriv}
	Let $(V^{2n+1},\xi)$ be a contact manifold and $\CS_{\xistab}$ a conformal symplectic structure on $\xistab=\xi\oplus\varepsilon_\real^{2k}$ that extends $\CS_{\xi}$ on $\xi$ and such that $\varepsilon_\real^{2k}$ is the $\CS_{\xistab}$-orthogonal complement of $\xi$ in $\xistab$.
	Consider also $J$ an almost complex structure on $\xistab$ which is tamed by $\CS_{\xistab}$ and such that $(\xistab,J)$ is trivial, via an isomorphism $\Phi\co (\xistab,J) \rightarrow \varepsilon_\compl^{n+k}$ of complex vector bundles over $V$.
	Let finally $f\co (V,\xi) \rightarrow (V,\xi)$ be a contactomorphism, $\family=(\gamma,X_1,\ldots,X_n)$ a $Y$-family of Lagrangian frames for $\xi$ and $(e_1,\cdots,e_k)$ a global Lagrangian frame $\varepsilon_\real^{2k}$.
	\\
	If the $Y$-family of matrices associated via $\Phi$ to the stabilization $(f_*\family)_{stab}$ is not homotopic, as map $Y\rightarrow \GL_{n+k}(\compl)$, to the $Y$-family of matrices associated via $\Phi$ to the stabilization $\family_{stab}$, then $f$ is not contact-isotopic to the identity.
\end{lemma}

\begin{remark}
	\label{rmk:formal_bis}
	Notice that the invariant just described is of a completely formal nature.
	More precisely, given a reference $Y$-family of Lagrangian frames $\family=(\gamma,X_1\ldots,X_n)$, the space $\diffVxif$ of formal contactomorphisms $(f,F_t)$ acts on the space of families of Lagrangian frames by pushforward (at time $1$) as follows: 
	\begin{equation*}
		(f,F_t)_*\family = (f\circ\gamma, F_1(X_1),\ldots,F_1(X_n)) \text{ .}
	\end{equation*}
	Moreover, under the same hypothesis as in \Cref{LemmaObstrStabTriv} but with $(f,F_t)\in\diffVxif$ instead of $f\in\diffVxi$, one can prove that if the $Y$-family of matrices associated to $((f,F_t)_*\family)_{stab}$ is not homotopic to the one associated to $\family_{stab}$, then $[(f,F_t)]$ is not trivial in $\pi_0(\diffVxif)$.
	
	\noindent
	As a consequence, such invariant detects non--triviality of contactomorphisms directly in $\pi_0(\diffVxif)$ (cf.\ \Cref{rmk:formal}), hence is not able to detect \emph{rigid} contactomorphisms, i.e.\ those which are non--trivial in $\pi_0(\diffVxi)$ but trivial in $\pi_0(\diffVxif)$.
\end{remark}


\section{Contact mapping classes of infinite order}
\label{SecProofs}

The aim of this section is to prove \Cref{ThmMain}, stated in the introduction.
We briefly recall the setting, reinterpreting it in terms of the construction of the double of a Liouville manifold from \Cref{SubSecProdDoubleLiouvDomWithCircle} and specifying explicitly who is the candidate contactomorphism $\Psi_c$. 

The starting data is that of a Weinstein manifold $(F^{2n-2},\omega_F,Z_F,\psi_F)$ such that there exists a $J_F$ tamed by $\omega_F$ for which $(TF,J_F)$ is trivial as complex vector bundle over $F$.
We then consider the stabilization $(F\times\compl,\omega=\omega_F\oplus\omega_0,Z=Z_F+Z_0,\psi=\psi_F+\abs{.}^2_\compl)$, with $\omega_0=rdr\wedge d\varphi$ and $Z_0 = \frac{1}{2}r\partial_r$, where $z=re^{i\varphi}\in\compl$. 
Then, for a regular value $c> \min \psi_F$ of $\psi\coeq\psi_F+\abs{.}^2_\compl$, we denote $W=\psi^{-1}((-\infty,c])$.
\\
One then considers the Weinstein manifold $(F\times\compl\times\real_s\times\cercle_\theta,\omega^\double,Z^\double,\psi^\double)$, where $\omega^\double = \omega_F + \omega_0 + ds\wedge d\theta$, $Z^\double = Z_F + Z_0 + s\partial_s$ and $\psi^\double(q,z,s,\theta) = \psi(q,z) + s^2$.  
Let also $\lambda=\iota_Z\omega$ on $\What=F\times\compl$ and $\lambda^\double = \lambda + s d\theta$ on $\What\times\real\times\cercle$.
The manifold of interest is then $(\psi^\double)^{-1}(c)=\doublefunW{\psi-c}\times\cercle$, equipped with the contact structure $\xi=\ker\alpha_{\psi-c}$, where $\alpha_{\psi-c} = \lambda^\double\vert_{(\psi^\double)^{-1}(c)}$.
We also have a simple diffeomorphism  $\Psi$ of $\doubleW\times\cercle$, given by the restriction of the diffeomorphism $\widehat{\Psi}$ of $F\times\compl\times\real\times\cercle$ defined as $\widehat{\Psi}(q,z,s,\theta)= (q,e^{i\theta}z,s,\theta)$.
\\
What we are going prove is that $\Psi$, is isotopic, via the flow $\psi_{Y}^t$ of the vector field $Y$ described in \Cref{LemmaContactomorphismProdWithCircle}, to a contactomorphism $\Psi_c=\Psi \circ \psi_Y^1$ which is of infinite order in $\pi_0\left(\diff{\doublefunW{\psi-c}\times\cercle,\alpha_{\psi-c}}\right)$.

\paragraph*{A complex trivialization.}
We start by extending the trivialization of $(TF,J_F)$ to a natural complex trivialization of the tangent bundle of $F\times\compl\times\real\times\cercle$.

More precisely, the almost complex structure $J\coeq J_F\oplus i$ on $F\times\compl$ can be further extended to $J^\double$ on $F\times\compl\times\real_s\times\cercle_\theta$ by defining $J^\double(\partial_s)\coeq \partial_\theta$ on $T(\real_s\times\cercle_\theta)$. 
Notice in particular that $J^\double$ is tamed by $\omega^\double$.
Let now $\varepsilon_F^{n-1}$ be the trivial complex vector bundle $(F\times\compl^{n-1},J_{std})$ over $F$ and denote $\nu\co(TF,J_F)\xrightarrow{\sim} \varepsilon_F^{n-1}$ the complex trivialization from \Cref{ThmMain}. 
Then, $\nu$ extends to a trivialization 
\begin{equation}
\label{EqTriv}
\mu\co\left(T\left(F\times\compl\times\real\times\cercle\right),J^\double\right)\xrightarrow{\sim}\varepsilon_{F\times\compl\times\real\times\cercle}^{n+1}
\end{equation}
such that, for each $(q,z,s,\theta)\in F\times\compl\times\real\times\cercle$, one has:
\begin{itemize}
	\item the following diagram commutes
	\begin{equation*}
	\begin{tikzcd}
	(T_qF,J_F) \ar["\nu_q",d] \ar["i",r]& \left(T_{(q,z,s,\theta)}\left(F\times\compl\times\real\times\cercle\right),J^\double\right) \ar["\mu_{(q,z,s,\theta)}", dd] \\
	(\varepsilon_F^{n-1})_q=\compl^{n-1} \ar["j", d] &\\
	(\varepsilon_F^{n+1})_{q}=\compl^{n+1} \ar["\Id", r] &
	(\varepsilon_{F\times\compl\times\real\times\cercle}^{n+1})_{(q,z,s,\theta)}=\compl^{n+1}
	\end{tikzcd}
	\end{equation*}
	where $i$ and $j$ are the natural inclusions given by $T_q F = T_qF \oplus\{(0,0,0)\}\subset T_{(q,z,s,\theta)}\left(F\times\compl\times\real\times\cercle\right)$ and $\compl^{n-1}=\compl^{n-1}\times\{(0,0)\}\subset \compl^{n+1}$;
	\item $\mu_{(q,z,s,\theta)}(\partial_x)=(0,\ldots,0,1,0)\in \compl^{n+1}$, 
	where we use here coordinates $(x,y)$ on the factor $\compl$ of  $F\times\compl\times\real\times\cercle$,
	\item $\mu_{(q,z,s,\theta)}(\partial_s)=(0,\ldots,0,1)\in \compl^{n+1}$, where $s$ is the coordinate on the factor $\real$ of $F\times\compl\times\real\times\cercle$.
\end{itemize}

\paragraph*{Flattening the hypersurface.}
We now need to modify the equation $\psi^\double - c = 0$ defining $\doubleW\times\cercle$, in order to ``flatten'' a portion of this hypersurface.
The reason for such modification is that we want to understand the flow of the vector field $Y$ from \Cref{LemmaContactomorphismProdWithCircle} in order to use the invariant defined in \Cref{SubSecFamLegBases} (cf. \Cref{rmk:flattening}): having partially flattened the hypersurface will allow us to give an explicit and easy expression for such flow, at least in the flattened region.
Of course, we also need to prove that one can indeed do such modification without altering the conclusions of \Cref{ThmMain}; this is done in \Cref{RmkChoiceRegEq} below. 

Let then $a >0 $ be very small; in particular, it can be chosen to be smaller than $\frac{c- \min(\psi_F)}{2}>0$ (this parameter will intervene later in the proof) and such that $\psi$ has no critical value between $c-2a$ and $c$.  
Consider also a non-decreasing smooth cut-off function $\chi\co \real\rightarrow[-1,1]$, equal to $1$ exactly on $(2a,+\infty)$, equal to $-1$ exactly on $(-\infty,-2a)$, and such that $\chi(x)=x$ for $x\in(-a,a)$.
Then, the function $f\co F\times\compl\rightarrow \real$ defined by $f\coeq \chi (\psi-c)$ is a regular equation of $M=f^{-1}(0)=\psi^{-1}(c)$; in particular, $Z^\double=Z+s\partial_s$ on $\What\times\real_s$ is transverse to $\doublefunW{f}$ too.
See \Cref{FigTwoDoubles}.
\begin{figure}[htb]
	\centering
	\def\svgwidth{180pt}
	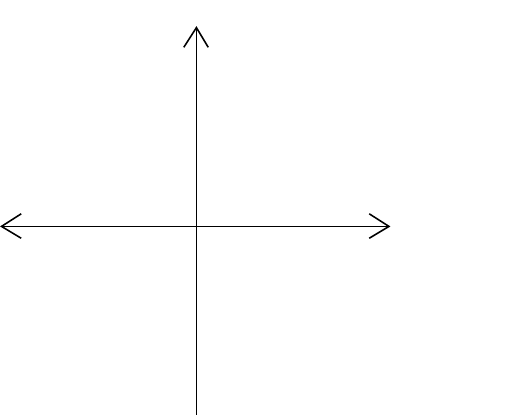
	\caption{$\doublefunW{f}$ and $\doublefunW{\psi-c}$ inside $\What\times\real_s$ and the vector field $Z^\double$, transverse to both; $W_-\coeq\doublefunW{f}\cap\{s=-1\}$ that appears in \ref{Step1} is also represented.}
	\label{FigTwoDoubles}
\end{figure}

As announced above, one can indeed prove \Cref{ThmMain} using the equation $f$ instead of $\psi-c$ without loss of generality:
\begin{lemma}
	\label{RmkChoiceRegEq}
	If the conclusion of \Cref{ThmMain} holds with the special choice of equation $f$ for $\doublefunW{f}\subset F\times\compl\times\real$, then it holds also for $\doublefunW{\psi-c}$ defined by $\psi-c$ (i.e. as in the statement of \Cref{ThmMain}).
\end{lemma}
\begin{proof}[Proof (\Cref{RmkChoiceRegEq})]
	Let $f_1\coeq\psi-c$, $f_0\coeq f=\chi(\psi-c)$ and $f_t=tf_1+(1-t)f_0$.
	According to \Cref{LemmaUniqContStrDouble} (and the expression for the vector field from \Cref{LemmaGenUniqContStrDouble}), the flow $\psi_{X_{t}}^1$ of the vector field $X_t=\frac{f_1-f_0}{df_t^\double(Z^\double)} Z^\double$ 
	gives a contactomorphism from $(\doublefunW{f_1}\times\cercle,\ker(\alpha_{f_1}))$ to $(\doublefunW{f_0}\times\cercle,\ker(\alpha_{f_0}))$.
	In order to prove \Cref{RmkChoiceRegEq}, it's then enough to show that the diffeomorphism $\psi_{X_{t}}^1\circ\Psi\circ (\psi_{X_{t}}^1)^{-1}$ of $\doublefunW{f_0}$ is still induced by the diffeomorphism of $F\times\compl\times\real\times\cercle$ given by $(q,z,s,\theta)\mapsto (q,e^{i\theta}z,s,\theta)$.
	But this is indeed the case, because the flow $\psi_{X_{t}}^1$ preserves the angular component of the $\compl$-factor as well as the $\cercle$-factor of the product $F\times\compl\times\real\times\cercle$, and hence commutes with $(q,z,s,\theta)\mapsto (q,e^{i\theta}z,s,\theta)$.
\end{proof}

\paragraph*{Proving non-triviality of $\Psi_c^k$.}
We know from \Cref{SubSecProdDoubleLiouvDomWithCircle} that, inside the Liouville manifold $(F\times\compl\times \real_s\times\cercle_\theta,  \lambda^\double=\lambda+sd\theta)$, the preimage of $(-\infty,0]$ via $F\times\compl\times \real_s\times\cercle_\theta\rightarrow \real$, $(q,z,s,\theta)\mapsto s^2+f(q,z)$, gives a Liouville filling of $(V=\doublefunW{f}\times\cercle_\theta,\alpha_f)$.
In particular, we have a natural isomorphism of real vector bundles $\xistab=\xi\oplus\varepsilon_V^2 \simeq T(F\times\compl\times\real\times\cercle)\vert_V$ over $V=\doublefunW{f}\times\cercle$, where $\varepsilon_V$ is the trivial real line bundle $V\times\real\to V$, given by the natural inclusion on $\xi$ and by sending the two canonical sections $e_1,e_2$ of $\varepsilon_V^2$ to the Liouville vector field and Reeb vector field respectively.
Moreover, under this isomorphism, $\omega^\double$ and $J^\double$ on $F\times\compl\times\real\times\cercle$ give, respectively, a conformal symplectic structure $\CS_{\xistab}$ on $\xistab$, which restricts to $\CS_{\xi}$ on $\xi$ and such that $\varepsilon_V^{2}$ is the $\CS_{\xistab}$-orthogonal complement of $\xi$, and a complex structure $J$ on $\xistab$ tamed by $\CS_{\xistab}$.
Lastly, the trivialization $\mu$ described in \Cref{EqTriv} gives, via the isomorphism just described, a trivialization of $\xistab$, which we still denote by $\mu$.
We are then in the setting of \Cref{SubSecFamLegBases} and can use of the invariant described there in order to prove the non-triviality of the contactomorphism $\Psi_c^k$.

More precisely, as already announced at the beginning of the section, the candidate contactomorphism is $\Psi_c = \Psi \circ \psi_Y^1$, with $\psi_Y^t$ the flow at time $t$ of $Y$ defined in \Cref{LemmaContactomorphismProdWithCircle}. 
In order to show that, for each $k\neq 0$, $\Psi_c^k$ is not contact isotopic to the identity, we are going to proceed by steps as follows:
\begin{enumerate}[align=left, leftmargin=*, label=Step \arabic*.]
	\item \label{Step1} Let $W_-\coeq \doublefunW{f}\cap\{s=-1\}\subset F\times\compl\times\real$ (see \Cref{FigTwoDoubles}). 
	Notice that, by definition of $a$ and construction of $f$, $W_-$ is just a (slightly shrinked) copy of $W$ inside $F\times\compl\times\real$, namely $\{\psi\leq c-a, s=-1\}\simeq W=\{\psi\leq c\}$ (recall $\psi$ has no critical values between $c-a$ and $c$). 
	We then describe an explicit $\cercle$-family of Lagrangian frames $\family$ for $\ker(\alpha_f)$ on $W_-\times\cercle$. 
	\item \label{Step2} We remark that, for all $t\geq0$, $\psi_{Y}^{t}(W_-\times\cercle)\subset W_-\times\cercle$, and we describe the behavior of the restriction of $\Psi_c$, and its iterates, to $W_-\times\cercle$.
	This allows us to describe, for all $k\geq 1$, the pushforward $(\Psi_c^k)_*\family$ of $\family$ via the $k$-th iterate of $\Psi_c$. 
	\item \label{Step3} We describe, for each $k\geq 0$, the family of matrices $B_k\co \cercle\rightarrow\GL_{n+1}(\compl)$ associated, via the trivialization $\mu$, to the stabilization $(\Psi_c^k)_*\family\oplus Z^\double$. 
	We then show that, if $k\geq1$, $B_k$ is not homotopically trivial as map $\cercle\rightarrow\GL_{n+1}(\compl)$.
\end{enumerate}
According to \Cref{LemmaObstrStabTriv}, this proves that, for all $k\geq 1$, the $k$-th iterate of the contactomorphism $\Psi_c$ is not contact isotopic to the identity. The space of contactomorphisms being a group, this implies the same conclusion for all $k<0$.

\paragraph{\ref{Step1}}
Let $q_0\in F$ be the global minimum of $\psi_F$, and $(v_1,\ldots,v_{n-1})$ be a Lagrangian frame of $(T_{q_0}F,\omega_F)$.
Let also $z_0=0\in\compl$ and $s_0=-1$. 
Notice in particular that, by the choice of $0<a<\frac{c-\min\psi_F}{2}$ and $\chi$ in the definition of $f$, the point $(q_0,z_0,s_0,\theta)\in F\times\compl\times\real\times\cercle$ actually lies in $W_-\times\cercle\subset\doublefunW{f}\times\cercle$.
Moreover, $\lambda^\double= - d\theta$ at the point $\gamma(\theta)$ (recall $Z_F$ is gradient-like for $d\psi_F$), so that
$\xi_f=\ker(\alpha_f)$ coincides, over the point $(q_0,z_0=0,s_0=-1,\theta)$, with the subspace $T_{q_0}F\oplus T_0\compl$ of $T_{(q_0,z_0,s_0,\theta)}\left(F\times\compl\times\real\times\cercle\right)$. 
In particular, if one considers the loop
\begin{align*}
	\gamma\co\cercle &\rightarrow W_-\times\cercle\subset\doublefunW{f}\times\cercle \\
	\theta&\mapsto(q_0,z_0=0,s_0=-1,\theta)
\end{align*}
then $\family\coeq\left(\gamma, v_1,\ldots,v_{n-1},\partial_x\right)$ is a $\cercle$-family of Lagrangian frames for $\ker(\alpha_f)$ over $\gamma$, where we denoted $z=(x,y)\in\compl$.

\paragraph{\ref{Step2}}
Here, we give explicit expressions for $\Psi_c^k$ and $(\Psi_c^k)_*\family$.
\begin{lemma}
	\label{Lemma1Step2}
	The time$-1$ flow $\psi_Y^1$ of $Y$ satisfies $\psi_Y^1(W_-\times\cercle)\subset W_-\times\cercle$ and, for each $k\geq 0$, $(\psi_Y^1)^k\vert_{W_-\times\cercle}$ has the following form:
	\begin{align*}
	(\psi_Y^1)^k \co & W_-\times\cercle \rightarrow W_-\times\cercle\\
	&(q,re^{i\varphi},-1,\theta)  \mapsto (Q_k(q,r),R_k(r) e^{i\varphi},-1,\theta)
	\end{align*}
	for some functions $Q\co F\times\compl \rightarrow F$ and $R\co \compl \rightarrow\real$ depending, respectively, only on $(q,r)$ and $r$, where we use polar coordinates $z=re^{i\varphi}\in\compl$. 
	Moreover, $R_k(0)=0$ and $Q_k(q_0,r)=q_0$ for every $r\geq0$
	\\
	In particular, $\Psi_c^k\circ\gamma=\gamma$ and, moreover, $\psi_Y^1$ commutes with $\Psi$ on the set $W_-\times\cercle$ (which is obviously preserved by $\Psi$ too), hence $\Psi_c^k = \Psi^k\circ(\psi_Y^1)^k$ on $W_-\times\cercle$.
\end{lemma}

\begin{lemma}
	\label{Lemma2Step2}
	Let $\gamma, (v_1,\cdots,v_{n-1})$ and $\family$ be as in \ref{Step1}. 
	Then, for each $k\geq 0$, there are a Lagrangian frame $(v_1^{k},\ldots,v_{n-1}^k)$ of $T_{q_0}F$ and a real number $s_k\neq 0$ such that $(\Psi_c^k)_*\family$ 
	is given by 
	\begin{equation*}
		\left(\,\gamma,\, v^k_1, \,\ldots,\,v_{n-1}^k,\, s_k \left[\cos(k\theta) \partial_x+\sin\left(k\theta\right)\partial_y\right]\,\right) \text{ .}
	\end{equation*}
\end{lemma}

\begin{proof}[Proof (\Cref{Lemma1Step2})]
	We give a proof by induction on $k$. 
	The case $k=0$ is trivial. 
	Moreover, it is not hard to check that if the statement of \Cref{Lemma1Step2} holds for both $k=N$ and $k=1$, then it also holds for $k=N+1$. 
	In other words, it's actually enough to show that the lemma holds for $k=1$.
	We then show that $\psi_Y^1$ can be written in the desired form, with $Q_1,R_1$ satisfying the desired properties. 
	
	For this, we use the formula for $Y$ given in \Cref{LemmaContactomorphismProdWithCircle}.
	Notice that, on $W_-\times\cercle$, the function $f$ is constant and the coordinate $s$ is constant at $-1$.
	Moreover, on $\What = F\times\compl$, the Liouville vector field $Z$ for $\lambda=\lambda_F + \frac{1}{2}r d\varphi$ is $Z_F+\frac{1}{2}r\partial_r$ and the vector field $X$ generating the circle action on the $\compl$-factor is just $\partial_\varphi$, using polar coordinates $z=re^{i\varphi}\in\compl$; in particular $\lambda(X)=\frac{1}{2}r^2$. 
	Hence, for all $(q,z,-1,\theta)\in W_-\times\cercle $, we have
	\begin{equation}
		\label{eq:vect_field_Y}
			Y (q,z,-1,\theta) 
			\; = \; 
			-\frac{r^2}{2} Z_F(q)-\frac{r^3}{4}\, \partial_r (re^{i\varphi})	
			\; = \; 
			-\frac{r^2}{2} \,Z (q,z) \text{ .}
	\end{equation}
	\\
	In particular, the flow $\psi_Y^t\co \doubleW\times\cercle\rightarrow\doubleW\times\cercle$ of $Y$ at time $t\geq 0$ then satisfies $\psi_Y^t(W_-\times\cercle)\subset W_-\times\cercle$. 
	Indeed, by the choice of the parameter $a$ in the definition of $f$, the restriction of $Y$ to $\partial (W_-\times\cercle)=\{\psi=c-a,s=-1\}$ is either $0$ (where $r=0$) or transverse to $\partial (W_-\times\cercle)$ (where $r\neq0$) and pointing in the direction of $W_-\times\cercle$.
	In other words, the orbit via the flow of $Y$ of each point of $W_-\times\cercle$ stays inside $W_-\times\cercle$ at all positive times.
%
	\\
	It also follows from \Cref{eq:vect_field_Y} that, at time $t=1$, the embedding $\psi_Y^1\co W_-\times\cercle\rightarrow W_-\times\cercle$ can be written as $\psi_Y^1(q,re^{i\varphi},-1,\theta)=(Q_1(q,r),R_1(r)e^{i\varphi},-1,\theta)$, for some functions $Q_1\co F\times\compl \rightarrow F$ and $R_1\co \compl \rightarrow\real$, with $Q_1$ and $R_1$ both independent of the angular component $\varphi$ on $\compl$. 
	Moreover, $\psi_Y^1$ clearly commutes with $\Psi$ on $W_-\times\cercle$ and, as $Y(q_0,z_0=0,-1,\theta)=0$ and $Z_F(q_0)=0$ (recall $Z_F$ is gradient--like for $\psi_F$), we also have $R_1(0)=0$ and $Q_1(q_0,r)=q_0$ for each $r\geq0$, as desired. 
	This concludes the proof of the base case $k=1$ of \Cref{Lemma1Step2}.
\end{proof}

\begin{proof}[Proof (\Cref{Lemma2Step2})]
	According to \Cref{Lemma1Step2}, $\Psi_c^k=\Psi^k\circ(\psi_Y^1)^k$ and $\Psi_c\circ\gamma=\gamma$.
	Notice that $d_{\gamma(\theta)}\Psi$ acts on $T_{\gamma(\theta)}(F\times\compl\times\real\times\cercle)$ as $\Id$ on the subspace $T_{q_0}F\oplus T_{s_0}\real\oplus T_{\theta}\cercle$ and as the rotation of angle $\theta$ on the subspace $T_{z_0=0}\compl$.
	It's then enough to prove that there are $v_1^k,\ldots,v_{n-1}^k \in T_{q_0}F$ and $s_k>0$ such that
	\begin{equation*}
		(\psi_Y^1)^k_{\;*}\family = \left(\,\gamma,\, v^k_1, \,\ldots,\,v_{n-1}^k,\, s_k \partial_x\,\right) \text{ .}
	\end{equation*}
	
	\noindent
	Using the expression for $\psi_Y^1$ from \Cref{Lemma1Step2} (recall in particular that $Q_k(q_0,r)=q_0$ for $r\geq0$), an explicit computation shows that $d_{\gamma(\theta)}\psi_{Y}^1 (\partial_x) = R_k'(0) \partial_x$; notice that $R_k'(0)$ is necessarily non-zero as $\psi_Y^1$ is a diffeomorphism.
	In particular, choosing $s_k\coeq R_k'(0)$ and $v_j^k\coeq d_{(q_0,0)}Q_k (v_j)$ for each $j=1,\ldots,{n-1}$, we get the desired expression for $(\psi_Y^1)^k_{\;*}\family$.
\end{proof}

\begin{remark}
	\label{rmk:flattening}
	The flattening procedure from $\doublefunW{\psi-c}\times\cercle$ to $\doublefunW{f}\times\cercle$ done at the beginning of the section finds its motivation in \Cref{Lemma1Step2,Lemma2Step2}. 
	Indeed, they both rely on the very explicit formula for $Y$ in the flattened picture, namely \Cref{eq:vect_field_Y}.
\end{remark}

\paragraph{\ref{Step3}} 
	We now consider the stabilization $(\Psi_c^k)_*\family \oplus Z^\double$ given, at each point $\gamma(\theta)$, by the following Lagrangian frame of $T_{\gamma(\theta)}(F\times\compl\times\real\times\cercle)$:
	\begin{equation*}
		\left(v^k_1,\ldots,v_{n-1}^k,  s_k \left[\cos(k\theta) \partial_x+\sin\left(k\theta\right)\partial_y\right], Z^\double (\gamma(\theta))\right) \text{ .}
	\end{equation*} 
	Notice that $\cos(k\theta) \,\partial_x+\sin(k\theta)\,\partial_y = \left[\cos(k\theta)+\sin(k\theta)\,J^\double\right]\partial_x$ and $Z(\gamma(\theta))=-\partial_s$.
	In particular, the family of matrices $B_k\co \cercle \rightarrow \GL_{n+1}(\compl)$ associated via the trivialization $\mu$ (defined in \Cref{EqTriv}) has the form
	\begin{equation*}
	B_k(\theta) = \left(
	\begin{array}{r@{}c|c@{}l}
	&    B_{0,k} & \mbox{0} \\\hline
	&    \mbox{0} &  
	\begin{matrix}\rule{0pt}{2ex}
	s_k e^{ik\theta} & 0 \\
	0 & -1
	\end{matrix}   
	\end{array} 
	\right)
	\end{equation*}
	where $B_{0,k}\in\GL_{n-1}(\compl)$ does not depend on $\theta$.
	
	Thus, $B_k$ is homotopically trivial as map $ \cercle \rightarrow \GL_{n+1}(\compl)$ if and only if $k = 0$.
	Indeed, $B_0$ is a constant map and $\det(B_k(\theta))=b_k e^{ik\theta}$, for a certain $b_k\in\compl\setminus\{0\}$ (notice that $b_k\neq 0$ necessarily because $B_k(\theta)\in GL_{n+1}(\compl)$), hence the map $\cercle\to\compl\setminus\{0\}$ defined as $\theta\mapsto \det(B_k(\theta))$ is homotopically non-trivial if $k\geq 1$.
	This concludes \ref{Step3}, hence the proof of \Cref{ThmMain}.

\bibliographystyle{halpha}
\bibliography{my_bibliography}


\end{document}